\documentclass[a4paper,10pt]{amsart}
\usepackage[utf8]{inputenc}
\usepackage{amsthm}
\usepackage{amsmath}
\usepackage{bm}
\usepackage{amsfonts}
\usepackage{amssymb}
\usepackage{mathtools}
\usepackage{mathrsfs}
\usepackage{setspace}
\usepackage{xcolor}
\usepackage{textgreek}
\usepackage{todonotes}
\usepackage{thmtools} 

\usepackage[pdfdisplaydoctitle,colorlinks,breaklinks,urlcolor=blue,linkcolor=blue,citecolor=blue]{hyperref} 

\newtheorem{thm}{Theorem}[section]

\newtheorem{definition}[thm]{Definition}

\newtheorem{lem}[thm]{Lemma}

\newtheorem{prop}[thm]{Proposition}

\newcommand{\N}{\mathbb{N}}

\newcommand{\PP}{\mathbb{P}}
\newcommand{\E}[1]{\mathbb{E}\left[#1\right]}
\newcommand{\Ee}[1]{\mathbb{E}_\varepsilon\left[#1\right]}
\newcommand{\Prob}[1]{\mathbb{P}\left(#1\right)}
\newcommand{\Probe}[1]{\mathbb{P}_\varepsilon\left(#1\right)}

\theoremstyle{remark}
\newtheorem{rmk}[thm]{Remark}

\numberwithin{equation}{section}

\title[LDP Stochastic Equations non-Lipschitz drift ]{Large Deviations for Stochastic equations in Hilbert Spaces with non-Lipschitz drift}
\author[U. Pappalettera]{Umberto Pappalettera}
  \address{Scuola Normale Superiore, Piazza dei Cavalieri, 7, 56126 Pisa, Italia}
  \email{\href{mailto:umberto.pappalettera@sns.it}{umberto.pappalettera@sns.it}}
\date\today

\begin{document}

\begin{abstract}
We prove a Freidlin-Wentzell result for stochastic differential equations in infinite-dimensional Hilbert spaces perturbed by a cylindrical Wiener process. We do not assume the drift to be Lipschitz continuous, but only continuous with at most linear growth. Our result applies, in particular, to a large class of nonlinear fractional diffusion equations perturbed by a space-time white noise.
\end{abstract}

\maketitle

\section{Introduction}

We investigate Large Deviations for the family, indexed by $\varepsilon>0$, of stochastic differential equations in a infinite dimensional separable Hilbert space $H$ 
\begin{equation} \label{eq:SDE}
dX^{\varepsilon}_t = \left( AX^{\varepsilon}_t + B(X^{\varepsilon}_t) \right) dt + \varepsilon dW_t, \quad t \in [0,T],
\end{equation}
with initial condition $X^{\varepsilon}_0=x_0$.
We assume that the linear operator $A:D(A) \subseteq H \to H$ is self-adjoint, with eigenvalues $0 > -\lambda_0 \geq \dots \geq -\lambda_n \geq \dots $ and such that $(-A)^{-1+\delta}$ is trace class for some $\delta \in (0,1)$, and the nonlinear term $B:H \to H$ is continuous with at most linear growth, namely $B \in C(H,H)$ and there exist finite constants $a,b$ such that $\|B(x)\|_H \leq a +b\|x\|_H$ for every $x \in H$. 
$W$ is a cylindrical Wiener process on $H$. 
As for the initial condition $x_0$, we assume $x_0 \in D((-A)^{\delta/2})$, see \autoref{sec:notation} for details. The condition on the initial datum is important in proving the so called \emph{exponential tightness} for solutions of (\ref{eq:SDE}), see \autoref{def:exptight} and \autoref{lem:exptight} below. 

In the present paper we interpret (\ref{eq:SDE}) in (probabilistically weak) mild formulation, see \cite[Theorem 13]{DPFlPrRo13} for well-posedness of the equation. 
Hence for every $\varepsilon>0$ there is a stochastic basis $(\Omega_\varepsilon, (\mathcal{F}^\varepsilon_t), \PP_\varepsilon)$ which supports a cylindrical Wiener process $W^\varepsilon$ and a process $X^\varepsilon$ solution of 
\begin{equation*}
X^{\varepsilon}_t = e^{tA}x_0 + \int_0^t e^{(t-s)A} B(X^{\varepsilon}_s)ds + \varepsilon \int_0^t e^{(t-s)A}dW^\varepsilon_s, \quad t \in [0,T].
\end{equation*} 
To ease notation, we drop the apex on $W^\varepsilon$ in the following. For brevity we denote by (H) the assumption:
\begin{itemize}
\item \textbf{(H)}: $A:D(A) \subseteq H \to H$ self-adjoint, with eigenvalues $0 > -\lambda_0 \geq \dots \geq \lambda_n \geq \dots $ and such that $(-A)^{-1+\delta}$ is trace class for some $\delta \in (0,1)$.
\end{itemize}

Our main result is the following:
\begin{thm} \label{thm:LDP}
Assume {(H)} and $B\in C(H,H)$ with at most linear growth. Let $\alpha \in (0,\delta/2)$. Then there exists $T>0$ sufficiently small such that a Large Deviation Principle on $C([0,T],D(-A)^{\alpha})$ holds for the family of laws $(\mathcal{L}(X^\varepsilon))_{\varepsilon>0}$, with rate $\varepsilon^2$ and action functional given by (\ref{eq:action2}). Moreover, if $B\in C_b(H,H)$, the thesis holds for every choice of $T < \infty$.
\end{thm}

Our result generalizes the celebrated theorem of Freidlin and Wentzell \cite{FrW12} to the infinite dimensional setting. As far as we know, there is no abstract result in infinite dimension covering the case of non-Lipschitz drift. On the other hand, the literature concerning Large Deviation Principles for particular equations with irregular drift is very rich; we mention among the others \cite{CW99,CeDe19,CeRo04,CM10,DM10}. 
However, in the aforementioned works it usually happens that the (deterministic) limit equation, obtained imposing $\varepsilon=0$, is well posed, which plays a central role in the proof of Large Deviation estimates. 
As a consequence, usually the action functional vanishes only at one point, that is the unique solution to the deterministic dynamics. 
On the contrary, in our work we do not necessarily have uniqueness for the unperturbed equation and so our action functional vanishes on a wide set, made of \emph{all} solutions to the deterministic equation. 
The same phenomenon also appears in \cite{Herr01,GHR01} for finite dimensional SDEs and in \cite{Ma10} for a class of SPDEs. 
In these works the authors study a second Large Deviation Principle, which selects, among all the solution of the unperturbed dynamics, some special solutions with additional properties. 
We do not treat this difficult problem in the abstract case, being it dependent of the particular equation under investigation, and we limit ourselves to establish the first Large Deviation Principle.

Our method, inspired by \cite{He01}, consists in the approximation of the nonlinear drift $B$ with a sequence of Lipschitz and bounded drifts, cfr. \autoref{thm:approx}. The approximation itself relies on the Kirszbraun extension Theorem. Once such approximation is given, one can prove Large Deviation estimates for the solutions of (\ref{eq:SDE}) using an auxiliary equation, with a more regular nonlinearity, for which Large Deviation estimates are easier to obtain. 

The paper is organized as follows. In \autoref{sec:LD} we recall definitions and first results concerning Large Deviations; in \autoref{sec:notation} we give a concise presentation of classical Large Deviation results for SDEs with regular coefficients; in \autoref{sec:approx} we prove \autoref{thm:approx}, which is the key to approximate the solution of (\ref{eq:SDE}) with more treatable processes; in \autoref{sec:main} we prove our main result, \autoref{thm:LDP}; in \autoref{sec:applications} we discuss some application of our results.

\section{Large Deviations} \label{sec:LD} 
For the convenience of the reader, here we give basic definitions about Large Deviations. Being only interested in Large Deviations for (\ref{eq:SDE}), we do not work in settings more general than we need and we refer to \cite{DeZe09,FrW12} for a specific discussion on the topic.
Let $(\mathcal{X},d)$ be a Polish space and let $(\mu^\varepsilon)_{\varepsilon>0}$ be a family of Borel probability measures on $(\mathcal{X},d)$. Suppose we are given the following two objects:
\begin{itemize}
\item a positive sequence $a_\varepsilon \to 0$ as $\varepsilon \to 0$, called \emph{rate};
\item a map $S : \mathcal{X} \to [0,\infty]$ with compact sublevels, called \emph{action functional}.
\end{itemize}

\begin{definition} \label{def:LDP}
We say that a Large Deviation Principle (in the following LDP) holds for the family $(\mu^\varepsilon)_{\varepsilon>0}$ with rate $a_\varepsilon$ and action functional $S$ if the following bounds are verified:
\begin{itemize}
\item \textbf{Lower Bound}: for any open $A \subseteq \mathcal{X}$
\begin{equation} \label{eq:LowerBound}
\liminf_{\varepsilon \to 0} a_\varepsilon \log \mu^\varepsilon (A) \geq -\inf_{x \in A} S(x);
\end{equation}
\item \textbf{Upper Bound}: for any closed $C \subseteq \mathcal{X}$
\begin{equation} \label{eq:UpperBound}
\limsup_{\varepsilon \to 0} a_\varepsilon \log \mu^\varepsilon (C) \leq -\inf_{x \in C} S(x).
\end{equation}
\end{itemize}
\end{definition}

\cite[Theorem 3.3]{FrW12} gives an useful criterion (that is actually an equivalent definition of LDP) to check lower bound, namely
\begin{itemize}
\item \textbf{Lower Bound II}: for every $\delta>0$, $\gamma>0$ and $x \in \mathcal{X}$ there exists $\varepsilon_0>0$ such that for every $\varepsilon>\varepsilon_0$
\begin{equation} \label{eq:LowerBound'}
\mu^\varepsilon \left( y \in \mathcal{X} : d(y,x)<\delta  \right) \geq \exp \left( - a_\varepsilon^{-1} [S(x)+\gamma] \right).
\end{equation}
\end{itemize}

The following property, called \emph{exponential tightness}, plays a central role throughout the paper.
\begin{definition} \label{def:exptight}
We say that the family $(\mu^\varepsilon)_{\varepsilon>0}$ is exponentially tight in $(\mathcal{X},d_{\mathcal{X}})$  (with rate $a_\varepsilon$) if for every $M < \infty$ there exist a compact set $K_M \subseteq \mathcal{X}$ such that 
\begin{equation*}
\limsup_{\varepsilon \to 0} a_\varepsilon \log \mu^\varepsilon \left( K_M^c \right) \leq -M.
\end{equation*}
\end{definition}

We conclude this section with two general results, called \emph{contraction principles}, which allow to transfer LDPs from one metric space to another.

\begin{thm}[Contraction Principle]\cite[Theorem 4.2.1]{DeZe09} \label{thm:CP} 
Let the family $(\mu^\varepsilon)_{\varepsilon>0}$ satisfy a LDP on a Polish space $(\mathcal{X},d_{\mathcal{X}})$ with rate $a_\varepsilon$ and action functional $S^\mathcal{X}$ and let $f$ be a continuous mapping from $(\mathcal{X},d_{\mathcal{X}})$ to another Polish space $(\mathcal{Y},d_{\mathcal{Y}})$. 
Define the push-forward measures $\nu^\varepsilon \coloneqq f_*(\mu^\varepsilon)$. Then the family $(\nu^\varepsilon)_{\varepsilon>0}$ satisfies a LDP on $(\mathcal{Y},d_{\mathcal{Y}})$ with rate $a_\varepsilon$ and action functional
\begin{equation*}
S^{\mathcal{Y}}(y) \coloneqq \inf_{x \in f^{-1}(y)}S^{\mathcal{X}}(x),
\end{equation*} 
the infimum over the empty set being equal to $+\infty$.
\end{thm}

\begin{thm}[Inverse Contraction Principle]\cite[Theorem 4.2.4]{DeZe09} \label{thm:ICP} 
Let $(\mathcal{X},d_{\mathcal{X}})$,  $(\mathcal{Y},d_{\mathcal{Y}})$ be Polish spaces and let $g:(\mathcal{X},d_{\mathcal{X}})\to(\mathcal{Y},d_{\mathcal{Y}})$ be a continuous injection.
If the family $(\mu^\varepsilon)_{\varepsilon>0}$ is exponentially tight in $(\mathcal{X},d_{\mathcal{X}})$ and the family $(\nu^\varepsilon)_{\varepsilon>0}$, $\nu^\varepsilon \coloneqq g_*(\mu^\varepsilon)$, satisfies a LDP on $(\mathcal{Y},d_{\mathcal{Y}})$ with rate $a_\varepsilon$ and action functional $S^{\mathcal{Y}}$, then the family $(\mu^\varepsilon)_{\varepsilon>0}$ satisfies a LDP on $(\mathcal{X},d_{\mathcal{X}})$ with rate $a_\varepsilon$ and action functional
\begin{equation*}
S^{\mathcal{X}} \coloneqq  S^{\mathcal{Y}} \circ g .
\end{equation*} 
\end{thm}

\section{Notation and preliminaries} \label{sec:notation}

In this section we recall some results concerning Large Deviations for SDEs with regular coefficients, perturbed by a cylindrical Wiener process. The results exposed here are classical, see \cite[Chapter 12]{DPZa14} for a reference. However, we prefer to give the proof of some well-known facts for the sake of presentation, and also because our setting is slightly different from the classical one, see \autoref{rmk:ICP} and \autoref{prop:LDPlip}. 

\subsection{Cylindrical Wiener process}
Let $H$ be a infinite dimensional separable real Hilbert space and let $(e_n)_{n\in\N}$ be a complete orthonormal system of $H$. On $H$ one can consider a cylindrical Wiener process $W$, formally given by the infinite sum
\begin{equation}\label{eq:W}
W_t = \sum_{n \in \N} \beta^n_t e_n,
\end{equation}
where $(\beta^n)_{n\in\N}$ is a sequence of independent standard Brownian motions on a common filtered probability space $(\Omega,(\mathcal{F}_t),\PP)$. It is well known (see \cite{DPZa14}) that the series \eqref{eq:W} converges in $L^2(\Omega,E)$ for some separable real Hilbert space $E$ larger than $H$, with Hilbert-Schmidt embedding $H \hookrightarrow E$. Moreover, by Kolmogorov continuity criterion, the process $W$ has trajectories in $C([0,T],E)$ almost surely.

An interesting feature of $W$ is the following: if one denotes by $\mu$ the law of $W$ on $C([0,T],E)$, then $\mu$ is a symmetric Gaussian measure on $C([0,T],E)$ with reproducing kernel $W^{1,2}_0([0,T],H)$, the space of Sobolev functions vanishing at  zero.
This fact is the key ingredient in proving the next result, which generalizes the celebrated Schilder's Theorem for Brownian motion in the infinite-dimensional setting.

\begin{prop}
The family of laws $(\mathcal{L}(\varepsilon W))_{\varepsilon > 0}$ satisfies a LDP on $C([0,T],E)$ with rate $\varepsilon^2$ and action functional
\begin{equation} \label{eq:LDP}
S_0(\varphi) \coloneqq \frac{1}{2} \|\varphi\|^2_{W^{1,2}_0([0,T],H)} = \frac{1}{2} \int_0^T \|\dot{\varphi}_t\|_H^2 dt,
\end{equation}
if $\varphi \in W^{1,2}_0([0,T],H)$ and $S_0(\varphi)=+\infty$ otherwise.
\end{prop}

\begin{proof}
It follows from \cite[Theorem 12.7]{DPZa14}. Notice that since the continuous embedding $W^{1,2}_0([0,T],H) \hookrightarrow C([0,T],E)$ is compact, the function $S_0$ has compact sublevels and is indeed an action functional. See also \cite{AIP13} for the same result in the stronger topology of H\"older continuous functions.
\end{proof}


\subsection{Stochastic convolution}
Let $W$ be a cylindrical Wiener process on $H$, defined on some filtered probability space $(\Omega,(\mathcal{F}_t),\PP)$. Define the stochastic convolution
\begin{equation*}
W_A(t) \coloneqq \int_0^t e^{(t-s)A}dW_s,
\end{equation*}
which is the unique (probabilistically strong) mild solution of the stochastic differential equation  
\begin{equation*}
d W_A(t) = A W_A(t) dt + dW_t, \quad t \in [0,T], \quad W_A(0)=0. 
\end{equation*}

It is well known that, under the assumption {(H)}, for every $\alpha \leq \delta/2$ the stochastic convolution $W_A$ has trajectories in $C([0,T],D((-A)^{\alpha})$ a.s., where $D((-A)^{\alpha})$ is a separable Hilbert space with norm $\|x\|_{D((-A)^{\alpha})} \coloneqq \|(-A)^{\alpha} x\|_H$. Moreover, the law of $W_A$ is a symmetric Gaussian measure on $C([0,T],D((-A)^{\alpha}))$ for every $\alpha \leq \delta/2$.

Fix $\alpha \in (0,\delta/2)$. Hereafter we denote 
\begin{equation*}
V \coloneqq D((-A)^{\alpha}), \quad V_0 \coloneqq D((-A)^{\delta/2}),
\end{equation*}
which are separable Hilbert spaces that enjoy compact embeddings $V_0 \hookrightarrow V \hookrightarrow H$.

\begin{prop} \label{prop:LDPWAH}
The family of laws $(\mathcal{L}(\varepsilon W_A))_{\varepsilon > 0}$  satisfies a LDP on $C([0,T],H)$ with rate $\varepsilon^2$ and action functional given by (\ref{eq:actionWA}).
\end{prop}

\begin{proof}
Let $\pi_N:H \to \pi_N H$ be the projection from $H$ onto the linear span of $e_1, \dots, e_N$, where $(e_n)_{n \in \N}$ is a basis of eigenvectors of $A$, and let $W_A^N \coloneqq \pi_N W_A$. Then $\varepsilon W_A^N$ satisfies the following equation in $C([0,T],\pi_N H)$:
\begin{equation*}
d \varepsilon W_A^N (t) = A \varepsilon W_A^N (t) dt + d \pi_N \varepsilon W_t, \quad t \in [0,T],
\end{equation*}
with initial condition $\varepsilon W_A^N (0)=0$. Since $A:\pi_N H \to \pi_N H$ is Lipschitz, mild solutions coincide with strong solutions, and by classical Freidlin-Wentzell Theorem for every integer $N$ a LDP holds for $(\mathcal{L}(\varepsilon W_A^N))_{\varepsilon>0}$ on $C([0,T],\pi_N H)$, with rate $\varepsilon^2$ and action functional given by
\begin{equation*}
S_N(\varphi) = \frac{1}{2} \int_0^T \| \dot{\varphi}_t  - A \varphi_t \|_H^2 dt, 
\end{equation*}
if $\varphi \in W^{1,2}_0 ([0,T],\pi_N H)$, and $S_N(\varphi)= +\infty $ otherwise. By \cite[Theorem 4.6.1]{DeZe09}, a LDP holds on $C([0,T],H)$ for $(\mathcal{L}(\varepsilon W_A))_{\varepsilon>0}$, with rate $\varepsilon^2$ and action functional 
\begin{equation} \label{eq:actionWA}
S(\varphi) = \sup_{N \in \N} S_N(\pi_N \varphi).
\end{equation}
\end{proof}

\autoref{prop:LDPWAH} above can be refined to obtain a LDP in the space $C([0,T],V)$ thanks to the following lemma. The validity of this stronger LDP will be very useful in the proof of \autoref{thm:LDP}.

\begin{lem} \label{lem:fernique}
Assume {(H)}. Then there exist constants $c,C>0$, depending only on $T,A$, such that for every $\varepsilon, R >0$
\begin{equation*}
\Prob{\| \varepsilon W_A \|_{C([0,T],V_0)} > R} \leq C \exp\left(- c \varepsilon^{-2} R^2 \right).
\end{equation*}
\end{lem}

\begin{proof}
Since the law of $W_A$ is a symmetric Gaussian measure on $C([0,T],V_0)$, by Fernique Theorem there exists $c>0$ such that
\begin{equation*}
\E{\exp\left( c \| W_A \|_{C([0,T],V_0)}^2 \right)} < \infty,
\end{equation*}
therefore by Markov inequality we have
\begin{gather*}
\Prob{\| \varepsilon W_A \|_{C([0,T],V_0)} > R} = \Prob{\exp\left( c \| W_A \|_{C([0,T],V_0)}^2 \right) > \exp\left( c \varepsilon^{-2} R^2 \right)} \\
\leq \exp\left(- c \varepsilon^{-2} R^2 \right) \E{\exp\left( c \| W_A \|_{C([0,T],V_0)}^2 \right)}\leq C \exp\left(- c \varepsilon^{-2} R^2 \right).
\end{gather*}
Notice also that the constant $c$ can be chosen to be decreasing with respect to $T$. This property plays a role in the detection of the constraint on $T$ in \autoref{thm:LDP}.
\end{proof}

\begin{prop} \label{rmk:ICP}
Assume {(H)}. Then the family of laws $(\mathcal{L}(\varepsilon W_A))_{\varepsilon > 0}$  satisfies a LDP on $C([0,T],V)$ with rate $\varepsilon^2$ and action functional given by (\ref{eq:actionWA}).
\end{prop}
\begin{proof}
Since the embedding $V_0 \hookrightarrow V$ is compact, \autoref{lem:fernique} above implies exponential tightness with rate $\varepsilon^2$ (in the sense of \autoref{def:exptight}) of the family $(\mathcal{L}(\varepsilon W_A))_{\varepsilon>0}$ in the space $C([0,T],V)$. Therefore by \autoref{thm:ICP} the LDP for the family $(\varepsilon W_A)_{\varepsilon>0}$ in the space $C([0,T],H)$ transfers back to a LDP in the space $C([0,T],V)$ with same rate and action functional.
\end{proof}

\subsection{SDEs with Lipschitz nonlinearity}

In this subsection we derive Freidlin-Wentzell Theorem for (probabilistically strong) mild solutions of the stochastic differential equation 
\begin{equation} \label{eq:SDElip}
dX^{\varepsilon}_t = \left( AX^{\varepsilon}_t + B(X^{\varepsilon}_t) \right) dt + \varepsilon dW_t, \quad t \in [0,T],
\end{equation}
with initial condition $X^{\varepsilon}_0=x_0 \in V_0$ and bounded Lipschitz nonlinearity $B$. The condition $x_0 \in V_0$ will be essential in the proof of forthcoming \autoref{lem:exptight}, although it is not necessary for \autoref{prop:3.5}.
We recall that in \cite[Chapter 7]{DPZa14} the authors prove there exists a unique mild solution of (\ref{eq:SDElip}) taking values in the space $C([0,T],H)$. The results stated here are classical and rely on Contraction Principle (\autoref{thm:CP}).

\begin{prop} \label{prop:3.5}
Assume (H) and $B \in Lip_b(H,H)$. For $\varepsilon>0$, let $X^{\varepsilon}$ be the unique mild solution of (\ref{eq:SDElip}). Then a LDP on the space $C([0,T],H)$ holds for the family of laws $(\mathcal{L}(X^{\varepsilon}))_{\varepsilon>0}$, with rate $\varepsilon^2$ and action functional given by (\ref{eq:actionLipSDE}).
\end{prop}

\begin{proof}
By definition of mild solution we have
\begin{equation*}
X^{\varepsilon}_t = e^{tA}x_0 + \int_0^t e^{(t-s)A} B(X^{\varepsilon}_s)ds + \varepsilon W_A(t).
\end{equation*}
Defining $\Gamma_B: C([0,T],H)\to C([0,T],H)$ to be the map that to a continuous function $w$ associates the unique solution $z$ of the equation
\begin{equation*}
z_t =e^{tA}x_0  + \int_0^t e^{(t-s)A}B(z_s) ds + w_t;
\end{equation*}
by Lipschitzianity of $B$ the map $\Gamma_B$ is continuous and invertible, with inverse:
\begin{equation*}
\Gamma_B^{-1}(z)_t = z_t - e^{tA}x_0 - \int_0^t e^{(t-s)A}B(z_s) ds.
\end{equation*}
Since $ X^{\varepsilon}= \Gamma_B( \varepsilon W_A)$, by \autoref{thm:CP} a LDP for $(\mathcal{L}(X^{\varepsilon}))_{\varepsilon>0}$ holds with rate $\varepsilon^2$ and action functional  
\begin{equation} \label{eq:actionLipSDE} 
S_B(\varphi) = S(\Gamma_B^{-1}(\varphi))= S \left( \varphi - e^{\cdot A}x_0  - \int_0^\cdot e^{(\cdot-s)A}B(\varphi_s) ds \right),
\end{equation}
where $S$ is given by (\ref{eq:actionWA}).
\end{proof}

As already done for the stochastic convolution, now we promote the LDP on $C([0,T],H)$ to a stronger LDP on $C([0,T],V)$. We prove first exponential tightness of the family of laws $(\mathcal{L}(X^\varepsilon))_{\varepsilon>0}$, using the following lemma.

\begin{lem} \label{lem:boundconv}
For every $\beta>0$ there exists a constant $C=C_\beta>0$ such that for every $x \in H$ and $t>0$
\begin{equation*}
\|(-A)^\beta e^{tA} x\|_H \leq \frac{C}{t^\beta}\|x\|_H.
\end{equation*}
\end{lem}

\begin{proof}
For every $x \in H$ one has $(-tA)^\beta e^{-tA} x = \sum_n t^\beta \lambda_n^\beta e^{-t \lambda_n} \langle x,e_n\rangle e_n$, where $(e_n)_{n \in \N}$ is a basis of eigenvectors of $A$. Since the real function $f(r) \coloneqq r^{\beta} e^{-r}$ is bounded from above by a constant $C_\beta$ uniformly in $r \geq 0$, the thesis follows.
\end{proof}

\begin{lem} \label{lem:exptight}
Assume (H) and $B \in Lip_b(H,H)$. Then the family of laws $(\mathcal{L}(X^\varepsilon))_{\varepsilon>0}$ is exponentially tight in $C([0,T],V)$ with rate $\varepsilon^2$.
\end{lem}

\begin{proof}
The thesis follows from the condition $x_0 \in V_0$, \autoref{lem:fernique} and the following inequality, holding for every $t \in [0,T]$
\begin{equation*}
\|X^\varepsilon_t\|_{V_0} \leq \|x_0\|_{V_0}  + C\|B\|_{\infty,H} + \|\varepsilon W_A(t)\|_{V_0},
\end{equation*} 
where $\|B\|_{\infty,H} = \sup_{x \in H} \|B(x)\|_H$ and $C=C_{A,T}$ is a constant independent of $\varepsilon$. Indeed we have
\begin{align*}
\|X^\varepsilon_t\|_{V_0} &\leq \|x_0\|_{V_0} + \int_0^t \left\| e^{(t-s)A}B(X^\varepsilon_s) \right\|_{V_0} ds +  \|\varepsilon W_A(t)\|_{V_0} \\
&= \|x_0\|_{V_0} + \int_0^t \left\| (-A)^{\delta/2}e^{(t-s)A}B(X^\varepsilon_s) \right\|_{H} ds +\|\varepsilon W_A(t)\|_{V_0} ,
\end{align*}
and by \autoref{lem:boundconv} the integral can be estimated with
\begin{gather*}
\int_0^t \left\| (-A)^{\delta/2}e^{(t-s)A}B(X^\varepsilon_s) \right\|_{H} ds \leq \int_0^t \frac{C_\delta}{(t-s)^{\delta/2}} \left\| B(X^\varepsilon_s) \right\|_{H} ds
\leq  C\|B\|_{\infty,H}.
\end{gather*}
\end{proof}

\autoref{lem:exptight} above allows to tranfer the LDP for the family $(\mathcal{L}(X^\varepsilon))_{\varepsilon>0}$ on the space $C([0,T],H)$ to a LDP on the space $C([0,T],V)$ with same rate and action functional, as stated in the following:

\begin{prop} \label{prop:LDPlip}
Assume (H) and $B \in Lip_b(H,H)$. For $\varepsilon>0$, let $X^{\varepsilon}$ be the unique mild solution of (\ref{eq:SDElip}). Then a LDP on the space $C([0,T],V)$ holds for the family  $(\mathcal{L}(X^{\varepsilon}))_{\varepsilon>0}$, with rate $\varepsilon^2$ and action functional given by (\ref{eq:actionLipSDE}).
\end{prop}

\section{Approximation of $B$} \label{sec:approx}

In this section we want to construct a sequence of functions $(B_R)_{R \in \mathbb{N}} \subseteq Lip_b(H,H)$ that approximates $B$ in a suitable sense, see \autoref{thm:approx} below.

Recall that the embedding $V \hookrightarrow H$ is compact, therefore $F_R \coloneqq \{\|x\|_{V} \leq R\}$ is compact in $H$ for every $R>0$.
Moreover, for every $x \in V$ it holds $\|x\|_H \leq \lambda_0^{-\alpha} \|x\|_V$. 
The main result of this section is the following:
\begin{thm} \label{thm:approx}
Assume {(H)} and $B\in C(H,H)$ with at most linear growth, i.e. there exist finite constants $a,b$ such that $\|B(x)\|_H \leq a + b\|x\|_H$ for every $x \in H$. Then there exist $C>0$ and a sequence $(B_R)_{R \in \mathbb{N}} \subseteq Lip_b(H,H)$ such that for every $R \geq 1$ one has $\|B_R\|_{\infty} \leq a+\lambda_0^{-\alpha}bR+1$ and
\begin{equation*}
\sup_{x \in F_R} \left\| {B}_R(x) - B(x) \right\|_H \leq \frac{C}{R^2} .
\end{equation*}
\end{thm}

\begin{proof}
Fix an integer $R \geq 1$. Without loss of generality, assume $a,b>0$. We artificially bound $B$ in the following way: take a smooth cut-off function $\rho = \rho_R: [0,\infty) \to [0,1]$ such that $\rho(r)= 1$ if $r \leq a+\lambda_0^{-\alpha}bR$ and  $\rho(r)= 0$ if $r \geq a+\lambda_0^{-\alpha}bR+1$, and consider from now on $B'_R(x)=\rho(\|B(x)\|_H) B(x)$. By the assumptions on the growth of $B$, $B'_R$ coincides with $B$ in $F_R$: indeed, if $\|x\|_V \leq R$ then also $\|x\|_H \leq \lambda_0^{-\alpha}R$, and by the linear growth assumption on $B$ we have $\|B(x)\|_H \leq a+\lambda_0^{-\alpha}bR$ and therefore $\rho(\|B(x)\|_H)=1$.

Since $B'_R \in C(H,H)$, the restriction of $B'_R$ to $F_{R+1}$ is uniformly continuous, and let $\delta_R$ be a positive number such that $\|B'_R(x)-B'_R(y)\|_H < 1/R^2$ for every $x,y \in F_{R+1}$, $\|x-y\|_H < \delta_R$. Without any loss of generality, we can also suppose $\delta_R < 1/{R}$. Denote also $\tau_R = \delta_R^p$, for some sufficiently large parameter $p$ to be determined later. 
Let $\nu_R \coloneqq \mathcal{L}(W_A(\tau_R))$ be the law on $H$ of the random variable $W_A(\tau_R)$, and define
\begin{equation*}
\overline{B}_R(x) \coloneqq \int_H B'_R(e^{A \tau_R}x+y) \nu_R(dy).
\end{equation*}

It is well known (see \cite[Theorem 6.2.2]{DPZa02}) that with this construction $\overline{B}_R \in C^\infty_b(H,H)$ for every integer $R$.
Moreover, $\overline{B}_R$ approximates $B'_R$ (and hence also $B$) uniformly on $F_R$: indeed, let $x \in F_R$ and consider
\begin{align*}
\overline{B}_R(x)-B'_R(x) &=  \int_{\substack{\|y\|_H < \delta_R \\ \|y\|_{V} \leq 1}} \left( B'_R(e^{A \tau_R}x+y)-B'_R(x)\right) \nu_R(dy)  \\
&+ \int_{\substack{\|y\|_H < \delta_R \\ \|y\|_{V} > 1}}  \left( B'_R(e^{A \tau_R}x+y)-B'_R(x)\right) \nu_R(dy)  \\ 
&+  \int_{\|y\|_H \geq \delta_R } \left( B'_R(e^{A \tau_R}x+y)-B'_R(x)\right) \nu_R(dy). 
\end{align*}

For any $R \geq 1$, the last two summands in the expression above are easily bounded with the following quantities
\begin{gather*}
\left\| \int_{\substack{\|y\|_H < \delta_R \\ \|y\|_{V} > 1}} \left( B'_R(e^{A \tau_R}x+y)-B'_R(x)\right) \nu_R(dy) \right\|_H \leq  2(a+\lambda_0^{-\alpha}bR+1) \nu_R \left( \|y\|_{V} > 1 \right), \\
\left\| \int_{\|y\|_H \geq \delta_R } \left( B'_R(e^{A \tau_R}x+y)-B'_R(x)\right) \nu_R(dy) \right\|_H 
\leq 2(a+\lambda_0^{-\alpha}bR+1) \nu_R\left( \|y\|_H \geq \delta_R \right),
\end{gather*}
which are smaller than $C/R^2$ for some constant $C>0$ independent of $R$. To prove this claim, it is sufficient to show $\nu_R \left( \|y\|_{V} > 1 \right)=\mathbb{P} \left( \|W_A(\tau_R)\|_V>1\right)\leq C/R^3$ and $\nu_R\left( \|y\|_H \geq \delta_R \right)=\mathbb{P} \left( \|W_A(\tau_R)\|_H \geq \delta_R\right) \leq C/R^3$.
For the first bound, by Markov inequality and we have
\begin{align*}
\mathbb{P} \left( \|W_A(\tau_R)\|_V>1\right) 
&\leq
\mathbb{E} \left[ \|(-A)^{\alpha}W_A(\tau_R)\|_H^2 \right]
=
\mathbb{E} \left[ \sum_{k \in \mathbb{N}} \left|\int_0^{\tau_R}\lambda_k^{\alpha} e^{-\lambda_k(\tau_R -s)} dW^k_s \right|^2 \right]
\\
&=
\frac12 \sum_{k \in \mathbb{N}} \lambda_k^{2\alpha-1} (1-e^{-2\lambda_k \tau_R})
\\
&=
\frac12 \sum_{k \in \mathbb{N}} \lambda_k^{\delta-1} \frac{1-e^{-2\lambda_k \tau_R}}{\lambda_k^{\delta-2\alpha}}
\leq
C\, \tau_R^{\delta-2\alpha},
\end{align*}
where in the last line we have used $\alpha<\delta/2$, $Tr((-A)^{-1+\delta}) < \infty$ and the inequality $1-e^{-r} \leq r$ for every $r \geq 0$. Since $\tau_R^{\delta-2\alpha} = \delta_R^{p(\delta-2\alpha)} \leq R^{-p(\delta-2\alpha)}$, the bound $\nu_R \left( \|y\|_{V} > 1 \right)\leq C/R^3$ holds taking $p$ sufficiently large.

As for the second bound, we argue in a similar fashion to obtain
\begin{align*}
\mathbb{P} \left( \|W_A(\tau_R)\|_H>\delta_R\right) 
&\leq
\frac{1}{\delta_R^2}
\mathbb{E} \left[ \|W_A(\tau_R)\|_H^2 \right]
=
\frac{1}{\delta_R^2}
\mathbb{E} \left[ \sum_{k \in \mathbb{N}} \left|\int_0^{\tau_R}e^{-\lambda_k(\tau_R -s)} dW^k_s \right|^2 \right]
\\
&=
\frac{1}{2\delta_R^2}
\sum_{k \in \mathbb{N}} \lambda_k^{-1} (1-e^{-2\lambda_k \tau_R})
\\
&=
\frac{1}{2\delta_R^2} 
\sum_{k \in \mathbb{N}} \lambda_k^{\delta-1} \frac{1-e^{-2\lambda_k \tau_R}}{\lambda_k^\delta}
\leq
C\, \frac{\tau_R^\delta}{\delta_R^2}.
\end{align*}
Again, $\tau_R^\delta / \delta_R^2 = \delta_R^{p\delta-2} \leq R^{2-p\delta}$, and taking $p$ sufficiently large we get the desired estimate $\nu_R\left( \|y\|_H \geq \delta_R \right)\leq C/R^3$.

Finally, concerning the first term we have
\begin{gather*}
\left\|\int_{\substack{\|y\|_H < \delta_R \\ \|y\|_{V} \leq 1}} \left( B'_R(e^{A \tau_R}x+y)-B'_R(x)\right) \nu_R(dy) \right\|_H \\
\leq
\sup_{\substack{\|y\|_H < \delta_R \\ \|y\|_{V} \leq 1}} \left\| B'_R(e^{A \tau_R}x+y)-B'_R(x) \right\|_H.
\end{gather*}
Notice that $x,x+y,e^{A \tau_R}x+y \in F_{R+1}$ for every $x \in F_R$ and $\|y\|_V \leq 1$, and $\|e^{A \tau_R}x-x \|_H \leq \tau_R^{\alpha}\|x\|_V \leq \delta_R^{\alpha p}R \leq \delta_R$ for any $p$ sufficiently large;
therefore, by triangular inequality and the very definition of $\delta_R$ one has 
\begin{align*}
\sup_{\substack{\|y\|_H < \delta_R \\ \|y\|_{V} \leq 1}} \left\| B'_R(e^{A \tau_R}x+y)-B'_R(x) \right\|_H \leq C/R^2.
\end{align*}
Putting all together, one finally obtains
\begin{equation} \label{eq:unif}
\sup_{x \in F_R} \left\| \overline{B}_R(x) - B'_R(x) \right\|_H \leq \frac{C}{R^2} .
\end{equation}

Notice that since $F_R$ is compact in $H$ for every integer $R \geq 1$, the restriction of $\overline{B}_R$ to $F_R$ is Lipschitz for every $R$. Now we state a fundamental result which allows us to extend the restriction of $\overline{B}_R$ to $F_R$ to a Lipschitz function defined on the whole of $H$.

\begin{lem}[Kirszbraun extension Theorem] \label{lem:Kirszbraun}
\cite[Theorem 1.31]{Sc69}
Let $H$ be a Hilbert space, $E$ any subset of $H$, and $f:E\to H$ a Lipschitz map. Then $f$ can be extended to a Lipschitz map defined on all of $H$ with the same Lipschitz constant of $f$. 
\end{lem}

By \autoref{lem:Kirszbraun} the restriction of $\overline{B}_R$ to $F_R$ can be extended to a globally Lipschitz function $\tilde{B}_R$. Clearly also $\tilde{B}_R$ satisfies (\ref{eq:unif}). Now we artificially bound $\tilde{B}_R$ again, defining 
\begin{equation} \label{eq:BR}
B_R(x) \coloneqq \rho\left(\|\tilde{B}_R(x)\|_H\right) \tilde{B}_R(x).
\end{equation}

Clearly $B_R$ so defined is bounded by $a+\lambda_0^{-\alpha}bR+1$. To  check that $B_R$ is still a Lipschitz function, take $x,y \in H$ and consider $\| {B}_R(x) - {B}_R(y)\|_H$. If $x,y$ are such that both $\| \tilde{B}_R(x)\|_H$ and $\|\tilde{B}_R(y)\|_H$ are greater than $a+\lambda_0^{-\alpha}bR+1$ there is nothing to prove, and if both $\| \tilde{B}_R(x)\|_H$ and $\|\tilde{B}_R(y)\|_H$ are smaller than $a+\lambda_0^{-\alpha}bR+2$ then one has the following bound
\begin{align*}
\| {B}_R(x) - {B}_R(y)\|_H &\leq \| \tilde{B}_R(x) - \tilde{B}_R(y)\|_H \\
&+ \| \tilde{B}_R(x)\|_H \left( \rho\left(\|\tilde{B}_R(x)\|_H\right) - \rho\left(\|\tilde{B}_R(y)\|_H\right) \right) \\
&\leq \left( \| \tilde{B}_R\|_{Lip} +(a+\lambda_0^{-\alpha}bR+2)\| \rho\|_{Lip}\| \tilde{B}_R\|_{Lip} \right) \|x-y\|_H.
\end{align*}

We are left with the case where, let say, $x,y$ are such that $\| \tilde{B}_R(x)\|_H \leq a+\lambda_0^{-\alpha}bR+1$ and $\| \tilde{B}_R(y)\|_H \geq a+\lambda_0^{-\alpha}bR+2$. Then the Lipschitz property of $\tilde{B}_R$ implies the bound $ \|x-y\|_H \geq \| \tilde{B}_R\|_{Lip}^{-1}$ (of course unless $\| \tilde{B}_R\|_{Lip}=0$, occurrence which trivializes this construction). So in this last case one has
\begin{align*}
\| {B}_R(x) - {B}_R(y)\|_H &= \| {B}_R(x) - {B}_R(y)\|_H \frac{\|x-y\|_H}{\|x-y\|_H} \\
&\leq 2(a+\lambda_0^{-\alpha}bR+1) \| \tilde{B}_R\|_{Lip}\|x-y\|_H.
\end{align*}

Recalling (\ref{eq:unif}) and the fact that $B'_R$ coincides with $B$ on $F_R$, the proof is complete.
\end{proof}

\section{Main result} \label{sec:main}
In this section we prove \autoref{thm:LDP} for solutions of (\ref{eq:SDE})
\begin{equation*} 
dX^{\varepsilon}_t = \left( AX^{\varepsilon}_t + B(X^{\varepsilon}_t) \right) dt + \varepsilon dW_t, \quad t \in [0,T],
\end{equation*}
with initial condition $X^{\varepsilon}_0=x_0 \in V_0$, under the assumptions {(H)} and $B\in C(H,H)$ with at most linear growth. Although here our assumptions on the nonlinear drift $B$ are weaker, we recover the same expression for the action functional valid for the case $B \in Lip_b(H,H)$:
\begin{equation} \label{eq:action}
S_B(\varphi) = S \left( \varphi - e^{\cdot A}x_0 -\int_0^\cdot e^{(\cdot-s)A}B(\varphi_s) ds \right),
\end{equation}
where $S$ is given by (\ref{eq:actionWA}). The strategy of the proof is the following, and it is adapted from \cite{He01}. To prove the lower bound (\ref{eq:LowerBound}), we approximate the nonlinearity $B$ with a sequence $(B_R)_{R \in \mathbb{N}}$ given by \autoref{thm:approx} and use the fact that a LDP does hold for the family $(\mathcal{L}(X^{\varepsilon,R}))_{\varepsilon>0}$, where $X^{\varepsilon,R}$ is the (probabilistically strong) solutions of (\ref{eq:SDE}) with $B$ replaced by $B_R$, defined on the same stochastic basis $(\Omega_\varepsilon, (\mathcal{F}^\varepsilon_t), \PP_\varepsilon)$ which supports $X^\varepsilon$ and $W$. Then, by exponential tightness of the family $(\mathcal{L}(X^{\varepsilon}))_{\varepsilon>0}$ (\autoref{lem:exptight2}) we deduce compactness of the sublevels of $S_B$ by the validity of the lower bound. Finally, upper bound (\ref{eq:UpperBound}) is proved using the same calculations used for the lower bound and compactness of the sublevels of $S_B$.

\subsection{A remark on the expression of $S_B$}
In this subsection we prove an equivalent expression for the action functional $S_B$ given by (\ref{eq:action}), more similar in the spirit to the action functional given by the Freidlin-Wentzell Theorem for finite-dimensional diffusions (see \cite{FrW12}). 

\begin{prop} \label{prop:formulation_SB}
Let $\varphi \in C([0,T],H)$. Then the following formula holds:
\begin{equation} \label{eq:action2}
S_B(\varphi) = \frac{1}{2} \int_0^T \|\dot{\varphi}_t - A \varphi_t - B(\varphi_t) \|_H^2 dt, 
\end{equation}
if $\varphi \in W^{1,2}([\tau,T],H) \cap L^2([\tau,T],D(A))$ for every $\tau>0$, $\varphi(0)=x_0$ and $S_B(\varphi) = +\infty$ otherwise.
\end{prop}
 
The main issue in proving the formula above consists in the identification of the domain of finiteness of $S_B$, that is the content of the forthcoming:

\begin{lem} \label{lem:domain}
If $S_B(\varphi)<+\infty$ then $\varphi \in W^{1,2}([\tau,T],H) \cap L^2([\tau,T],D(A))$ for every $\tau>0$, $\dot{\varphi} - A\varphi - B(\varphi) \in L^2([0,T],H)$ and $\varphi(0)=x_0$.
\end{lem}

\begin{proof}
To ease notations define $\psi_t \coloneqq \varphi_t - e^{tA}x_0 - \int_0^t e^{(t-s)A} B(\varphi_s)ds$. One has
\begin{equation*}
S_B(\varphi) = S(\psi) = \sup_{N \in \mathbb{N}} S_N(\pi_N \psi) < \infty,
\end{equation*}
and $\pi_N\psi \in W^{1,2}_0([0,T],\pi_N H)$ for every $N$ follows, from which one can deduce also $\pi_N \varphi \in W^{1,2}([0,T],\pi_N H)$ and $\varphi(0)=x_0$. Moreover, it is also easy to check that 
\begin{align*}
\frac{d}{dt} (\pi_N \psi_t) - A \pi_N \psi_t &= \frac{d}{dt} (\pi_N \varphi_t) - Ae^{tA}\pi_N x_0 - \pi_N B(\varphi_t) - A \int_0^t e^{(t-s)A}\pi_N B(\varphi_s) ds \\
&- A\pi_N \varphi_t + Ae^{tA}\pi_N x_0 +  A \int_0^t e^{(t-s)A}\pi_N B(\varphi_s) ds \\
&=\frac{d}{dt} (\pi_N \varphi_t) - A\pi_N \varphi_t- \pi_N B(\varphi_t) ,
\end{align*}
and therefore $\sup_{N \in \mathbb{N}} S_N(\pi_N \psi) < \infty$ implies
\begin{gather*}
\sup_{N \in \mathbb{N}} \int_0^T \left\|\frac{d}{dt} (\pi_N \varphi_t) - A\pi_N \varphi_t  - \pi_N B(\varphi_t) \right\|_H^2 dt < \infty.
\end{gather*}

By our assumption on the growth of $B$ it holds $B(\varphi) \in L^2([0,T],H)$, and from the line above one can deduce
\begin{gather} \label{eq:supN}
\sup_{N \in \mathbb{N}} \int_0^T \left\|\frac{d}{dt} (\pi_N \varphi_t) - A\pi_N \varphi_t\right\|_H^2 dt < \infty. 
\end{gather}

Call $g_{N}:\left[0,T\right]  \rightarrow \pi_N H$ the function 
\begin{equation*}
g_{N}\left(  t\right)  :=\frac{d}{dt}\left(  \pi_{N}\varphi\right)  \left(
t\right)  -A   \pi_{N}\varphi_t .
\end{equation*}
Since $\sup_{N \in \mathbb{N}} \left\| g_{N}\right\| _{L^{2}\left( [0,T],H\right)  }<\infty$ there exists a subsequence (that we still denote $(g_N)_{N \in \mathbb{N}}$) which converges weakly to some
$g$ in $L^{2}([0,T],H)$. 
From the identity
\begin{equation*} 
\frac{d}{dt}\left(  \pi_{N}\varphi\right)(t) = A 
\pi_{N}\varphi_t    + g_{N}\left(  t\right)
\end{equation*}
and thus 
\begin{equation*}
\langle \varphi_t , \pi_N h \rangle_H = \langle x_0 , \pi_N h \rangle_H + \int_0^t \langle \varphi_s , \pi_N Ah \rangle_H ds + \int_0^t \langle g_N(s) , h \rangle_H ds 
\end{equation*}
for every $h\in D\left(  A\right)  $, we deduce by Lebesgue Theorem
\begin{equation*}
\langle \varphi_t , h \rangle_H = \langle x_0 , h \rangle_H +\int_0^t \langle \varphi_s ,  Ah \rangle_H ds + \int_0^t \langle g(s) , h \rangle_H ds. 
\end{equation*}

Since weak solutions are mild solutions, we deduce the expression of $\varphi$:
\[
\varphi_t  = e^{tA} x_0 + \int_{0}
^{t}e^{\left(  t-s\right)  A}g\left(  s\right)  ds
\qquad
\mbox{for a.e. } t \in [0,T],
\]
which implies $\varphi \in W^{1,2}([\tau,T],H) \cap L^2([\tau,T],D(A))$ for every $\tau>0$ and 
\[
\dot{\varphi}_t  = A \varphi_t +  g(t)
\qquad
\mbox{for a.e. } t \in [0,T],
\]
that means exactly $\dot{\varphi}  - A \varphi = g \in L^2([0,T],H)$. Since also $B(\varphi) \in L^2([0,T],H)$, the proof is complete.
\end{proof}

\begin{rmk} \label{rmk:domain}
If $x_0 \in D((-A)^{1/2})$ we are actually able to prove $\varphi \in W^{1,2}([0,T],H) \cap L^2([0,T],D(A))$ whenever $S_B(\varphi)<+\infty$: indeed, for every $N \in \mathbb{N}$ it holds
\begin{align*}
\int_0^T \left\|\frac{d}{dt} (\pi_N \varphi_t) - A\pi_N \varphi_t\right\|_H^2 dt
&=
\int_0^T \left\|\frac{d}{dt} (\pi_N \varphi_t)\right\|_H^2 dt
+
\int_0^T \left\|A\pi_N \varphi_t\right\|_H^2 dt
\\
&\quad+
2 \left\|(-A)^{1/2} \pi_N \varphi_T\right\|_H^2 
-
2 \left\|(-A)^{1/2} \pi_N \varphi_0\right\|_H^2
\\
&\geq
\int_0^T \left\|\frac{d}{dt} (\pi_N \varphi_t)\right\|_H^2 dt
+
\int_0^T \left\|A\pi_N \varphi_t\right\|_H^2 dt
\\
&\quad
-
2 \left\|x_0\right\|_{D((-A)^{1/2})}^2,
\end{align*}
and taking the supremum over $N$ we deduce by (\ref{eq:supN})
\begin{align} \label{eq:inequality_rmk}
\int_0^T \left\|\dot{\varphi}_t\right\|_H^2 dt
&+
\int_0^T \left\|A \varphi_t\right\|_H^2 dt
\\
&=
\sup_{N \in \mathbb{N}} \nonumber
\int_0^T \left\|\frac{d}{dt} (\pi_N \varphi_t)\right\|_H^2 dt
+
\int_0^T \left\|A\pi_N \varphi_t\right\|_H^2 dt
\\
&\leq \nonumber
\sup_{N \in \mathbb{N}}
\int_0^T \left\|\frac{d}{dt} (\pi_N \varphi_t) - A\pi_N \varphi_t\right\|_H^2 dt
+
2 \left\|x_0\right\|_{D((-A)^{1/2})}^2 < \infty.
\end{align} 
\end{rmk}

We are now able to prove \autoref{prop:formulation_SB}.
\begin{proof}[Proof of \autoref{prop:formulation_SB}]
By previous \autoref{lem:domain}, if $S_B(\varphi)<+\infty$ then $\varphi \in W^{1,2}([\tau,T],H) \cap L^2([\tau,T],D(A))$ for every $\tau>0$ and $\varphi(0)=x_0$.
To prove the validity of (\ref{eq:action2}), let $g=\dot{\varphi} - A \varphi \in L^2([0,T],H)$ be as in the previous lemma. 
We have the strong convergences in $L^2([0,T],H)$:
\begin{gather*}
\pi_N g \to g, \quad 
\pi_N B(\varphi) \to B(\varphi), 
\end{gather*}
and therefore
\begin{align*}
S_B(\varphi)
&=
\sup_{N \in \mathbb{N}} \int_0^T \left\|\frac{d}{dt} (\pi_N \varphi_t)- A\pi_N \varphi_t  - \pi_N B(\varphi_t) \right\|_H^2 dt 
\\
&=
\sup_{N \in \mathbb{N}} \int_0^T \left\|\pi_N g(t)- \pi_N B(\varphi_t) \right\|_H^2 dt
\\
&=
\int_0^T \left\|g(t)- B(\varphi_t) \right\|_H^2 dt
\\
&=
\int_0^T \left\|\dot{\varphi}_t - A \varphi_t - B(\varphi_t) \right\|_H^2 dt.
\end{align*}
\end{proof}

\subsection{Preliminary lemmas}

\begin{lem} \label{lem:exptight2}
Assume (H) and $B \in C(H,H)$ with at most linear growth: $\|B(x)\|_H \leq a+b\|x\|_H$ for every $x \in H$. Then the family of laws $(\mathcal{L}(X^\varepsilon))_{\varepsilon>0}$ is exponentially tight in $C([0,T],V)$ with rate $\varepsilon^2$.
\end{lem}
\begin{proof}
Arguing as in the proof of \autoref{lem:exptight}, one has
\begin{align*}
\|X^\varepsilon_t\|_{V_0} &\leq \|x_0\|_{V_0} + \int_0^t \left\| e^{(t-s)A}B(X^\varepsilon_s) \right\|_{V_0} ds +  \|\varepsilon W_A(t)\|_{V_0} \\
&= \|x_0\|_{V_0} + \int_0^t \left\| (-A)^{\delta/2}e^{(t-s)A}B(X^\varepsilon_s) \right\|_{H} ds +\|\varepsilon W_A(t)\|_{V_0} ,
\end{align*}
and by \autoref{lem:boundconv} the integral can be estimated with
\begin{align*}
\int_0^t \left\| (-A)^{\delta/2}e^{(t-s)A}B(X^\varepsilon_s) \right\|_{H} ds &\leq \int_0^t \frac{C_\delta}{(t-s)^{\delta/2}} \left\| B(X^\varepsilon_s) \right\|_{H} ds \\
&\leq  \int_0^t \frac{C_\delta}{(t-s)^{\delta/2}} \left( a +b\|X^\varepsilon_s\|_H \right) ds.
\end{align*}

Hence we have for some constant $C=C_{A,B,T}$ the inequality
\begin{align*}
\|X^\varepsilon_t\|_{V_0} &\leq \|x_0\|_{V_0} + \|\varepsilon W_A(t)\|_{V_0} + C + \int_0^t \frac{C}{(t-s)^{\delta/2}} \|X^\varepsilon_s\|_{V_0} ds.
\end{align*}

By \cite[Theorem 1, Corollary 2]{Ye07} there exists another constant $C=C_{A,B,T}$ such that
\begin{align*}
\|X^\varepsilon_t\|_{V_0} &\leq C \left( \|x_0\|_{V_0} + \|\varepsilon W_A(t)\|_{V_0} + 1 \right),
\end{align*}
and the thesis follows by \autoref{lem:fernique}.
\end{proof}

\begin{lem}[Girsanov Formula] \label{lem:girs} \cite[Theorem 13]{DPFlPrRo13} Assume (H) and $B \in C(H,H)$ with at most linear growth.
Let $\tilde{\PP}_\varepsilon$ be the probability measure on $(\Omega_\varepsilon,(\mathcal{F}^\varepsilon_t))$ implicitly given by
\begin{equation*}
\PP_\varepsilon = \exp \left( \varepsilon^{-1} \int_0^T \langle B(X^\varepsilon_s),d\tilde{W}_s \rangle - \frac{\varepsilon^{-2}}{2}\int_0^T \|B(X^\varepsilon_s)\|^2_{H} ds \right) \tilde{\PP}_\varepsilon,
\end{equation*}
where $\tilde{W}_t \coloneqq W_t + \varepsilon^{-1} \int_0^t B(X^\varepsilon_s)ds$. Then $\mathcal{L}_{\tilde{\PP}_\varepsilon} (\tilde{W},X^\varepsilon) = \mathcal{L}_{{\PP}_\varepsilon} ({W, Z^\varepsilon})$, where $Z^\varepsilon_t = e^{t A}x_0 + \varepsilon W_A(t)$ is the stochastic convolution starting at $x_0$.
\end{lem}

\begin{lem}[Exponential trick] \label{lem:trick}
Let $W$ be a cylindrical Wiener process on $H$, defined on some filtered probability space $(\Omega,(\mathcal{F}_t),\PP)$ and let $Y \in L^{\infty}(\Omega,L^{\infty}([0,T],H))$ be a progressively measurable process. Then for every $c>0$ we have
\begin{align*}
\Prob{ \int_0^T \langle Y_s , dW_s \rangle > c } \leq \exp \left( - \frac{c^2}{2\|Y\|_{\infty}^2 T}\right).
\end{align*}
\end{lem}

\begin{proof}
Take a positive number $\lambda$ and rewrite the event we are interested in as
\begin{align*}
\left\{ \exp \left( \lambda \int_0^T \langle Y_s , dW_s \rangle  - \frac{\lambda^2}{2} \int_0^T \|Y_s\|_{H}^2 ds \right) > \exp \left( \lambda c - \frac{\lambda^2}{2}\int_0^T \|Y_s\|_{H}^2 ds \right) \right\}, 
\end{align*}
so that the desired probability is less or equal to the probability
\begin{align*}
\PP \left( \exp \left( \lambda \int_0^T \langle Y_s , dW_s \rangle  - \frac{\lambda^2}{2} \int_0^T \|Y_s\|^2 ds \right) > \exp \left( \lambda c - \frac{\lambda^2}{2} \|Y\|_{\infty}^2 T \right) \right).
\end{align*}

Since $\lambda Y$ is bounded it satisfies the Novikov's condition, therefore the LHS of the expression above is the value of a martingale at time $T$ and thus its expected value is equal to $1$. Markov inequality gives
\begin{align*}
\PP \left(  \int_0^T \langle Y_s , dW_s \rangle  > c \right) \leq \exp \left( - \lambda c + \frac{\lambda^2}{2} \|Y\|_{\infty}^2 T \right),
\end{align*} 
and choosing $\lambda = c/\|Y\|_{\infty}^2 T$ the thesis follows.
\end{proof}

\subsection{Proof of \autoref{thm:LDP}}

\begin{proof}[Proof of the lower bound]
We check lower bound in the formulation (\ref{eq:LowerBound'}). Fix $\delta>0$, $\gamma>0$ and $\varphi \in C([0,T],V)$.
Denote $U_\delta$ the open ball of radius $\delta$ centered in $\varphi$ with respect to the distance of $C([0,T],V)$.
By \autoref{lem:girs} we have the following identity
\begin{gather*}
\PP_\varepsilon \left( X^\varepsilon \in U_\delta \right)  = \tilde{\mathbb{E}}_\varepsilon \left[ \mathbf{1}_{\{X^\varepsilon \in U_\delta\}} \exp \left( \varepsilon^{-1} \int_0^T \langle B(X^\varepsilon_s),d\tilde{W}_s \rangle - \frac{\varepsilon^{-2}}{2}\int_0^T \|B(X^\varepsilon_s)\|^2_{H} ds \right) \right],
\end{gather*}
where $\tilde{\mathbb{E}}_\varepsilon$ stands for the expectation with respect to the probability $\tilde{\mathbb{P}}_\varepsilon$; a similar formula holds for $X^{\varepsilon,R}$. To ease notation, denote
\begin{gather*}
\xi_T \coloneqq \varepsilon^{-1} \int_0^T \langle B(Z^\varepsilon_s),dW_s \rangle - \frac{\varepsilon^{-2}}{2}\int_0^T \|B(Z^\varepsilon_s)\|^2_{H} ds, \\
\xi_T^R \coloneqq \varepsilon^{-1} \int_0^T \langle B_R(Z^\varepsilon_s),dW_s \rangle - \frac{\varepsilon^{-2}}{2}\int_0^T \|B_R(Z^\varepsilon_s)\|^2_{H} ds,
\end{gather*}
so that we have the following identities
\begin{gather*}
\PP_\varepsilon \left( X^\varepsilon \in U_\delta \right) = \Ee{\mathbf{1}_{\{Z^\varepsilon \in U_\delta \}} e^{\xi_T}}, \quad \PP_\varepsilon \left( X^{\varepsilon,R} \in U_\delta \right) = \Ee{\mathbf{1}_{\{Z^\varepsilon \in U_\delta \}} e^{\xi_T^R}}.
\end{gather*}

Now introduce the auxiliary sets, depending of parameters $\alpha>0$ and integer $R \geq 1$:
\begin{gather*}
E_\alpha \coloneqq \{ |\xi_T - \xi_T^R| > \alpha\},\quad G_R \coloneqq \{ \| Z^\varepsilon \|_{C([0,T],V)} \leq R-1 \}.
\end{gather*}

A simple computation (we refer to \cite[Proposition I.14]{He01} for the details omitted here) yields
\begin{gather*}
\Probe{X^\varepsilon \in U_\delta} \geq \Probe{X^{\varepsilon,R} \in U_\delta} e^{-\alpha} - \Probe{E_\alpha^c}^{1/2}\Ee{e^{2\xi_T^R}}^{1/2}.
\end{gather*}

Now take $\alpha = \varepsilon^{-2}\gamma$. Our next step is to prove that the second summand in the RHS above does not play any role in Large Deviations for the law of $X^\varepsilon$, namely it can be absorbed into the $\gamma$ when checking (\ref{eq:LowerBound'}). Let us estimate first the expected value
\begin{gather*}
\Ee{e^{2\xi_T^R}}^{1/2} = \Ee{\exp \left( 2\varepsilon^{-1} \int_0^T \langle B_R(Z^\varepsilon_s),d{W}_s \rangle - \varepsilon^{-2}\int_0^T \|B_R(Z^\varepsilon_s)\|^2_{H} ds \right) }^{1/2}.
\end{gather*}
Since $B_R$ is bounded, Novikov condition applies and therefore
\begin{gather*}
\Ee{\exp \left( 2\varepsilon^{-1} \int_0^T \langle B_R(Z^\varepsilon_s),d{W}_s \rangle - 2\varepsilon^{-2}\int_0^T \|B_R(Z^\varepsilon_s)\|^2_{H} ds \right) }^{1/2} =1.
\end{gather*}
We deduce the following bound
\begin{gather*}
\Ee{e^{2\xi_T^R}}^{1/2} \leq  \exp \left( \frac{\varepsilon^{-2}T(a+\lambda_0^{-\alpha}bR+1)^2}{2} \right) .
\end{gather*}

Regarding the other term, in general one has the inequality
\begin{gather*}
\Probe{E_{\varepsilon^{-2}\gamma}^c}^{1/2} \leq \Probe{E_{\varepsilon^{-2}\gamma}^c \cap G_R}^{1/2} + \Probe{E_{\varepsilon^{-2}\gamma}^c\cap G_R^c}^{1/2}. 
\end{gather*}
The second summand is easily controlled with \autoref{lem:fernique} 
\begin{gather*}
 \Probe{E_{\varepsilon^{-2}\gamma}^c\cap G_R^c} \leq \Probe{G_R^c}    \leq C \exp\left(- c \varepsilon^{-2} (R-1-\|x_0\|_{V_0})^2 \right).
\end{gather*}
To estimate the other term, notice that on $G_R$ we have $Z^\varepsilon_s \in F_{R-1} \subseteq F_R$ for every $s \in [0,T]$ and thus
\begin{align*}
\|B(Z^\varepsilon_s)\|^2_{H} - \|B_R(Z^\varepsilon_s)\|^2_{H} 
&= 
(\|B(Z^\varepsilon_s)\|_{H} + \|B_R(Z^\varepsilon_s)\|_{H})
(\|B(Z^\varepsilon_s)\|_{H} - \|B_R(Z^\varepsilon_s)\|_{H}) 
\\
&\leq 
2(a+\lambda_0^{-\alpha}bR+1)(\|B(Z^\varepsilon_s)\|_{H} - \|B_R(Z^\varepsilon_s)\|_{H})
\\
&\leq 
2(a+\lambda_0^{-\alpha}bR+1)\|B(Z^\varepsilon_s)-B_R(Z^\varepsilon_s)\|_{H}
\\
&\leq
\frac{2(a+\lambda_0^{-\alpha}bR+1)C}{R^2}.
\end{align*}
In particular, on $G_R$ we have
\begin{gather*}
\left| \frac{\varepsilon^{-2}}{2}\int_0^T \|B(Z^\varepsilon_s)\|^2_{H} - \|B_R(Z^\varepsilon_s)\|^2_{H} \right| \leq \frac{\varepsilon^{-2} T C (a+\lambda_0^{-\alpha}bR+1)}{R^2},
\end{gather*}
and therefore we can control the probability of the event $E_{\varepsilon^{-2}\gamma}^c \cap G_R$ for $R$ such that ${ 2 T C (a+\lambda_0^{-\alpha}bR+1)} < {\gamma R^2}$ simply with
\begin{gather} \label{eq:cap}
 \Probe{E_{\varepsilon^{-2}\gamma}^c \cap G_R} \leq \Probe{\left\{\left| \varepsilon^{-1}\int_0^T \langle B(Z^\varepsilon_s) -B_R(Z^\varepsilon_s) ,dW_s \rangle \right| > \frac{\varepsilon^{-2}\gamma}{2} \right\}  \cap G_R}.
\end{gather}
Let us now estimate the probability of the latter event.
Let $\rho:[0,\infty) \to [0,1]$ be a smooth cut-off function such that $\rho(r)=1$ if $r \leq R-1$, and $\rho(r)=0$ if $r \geq R$ (to avoid any confusion, let us point our that this $\rho$ is different from the cut-off used in \autoref{thm:approx}). 
The function $\rho(\|x\|_V) (B(x) -B_R(x))$, $x \in H$, coincides with $B(x) -B_R(x)$ for every $x \in F_{R-1}$, and by construction it is globally bounded in $H$ by the constant $C/R^2$.
As a consequence, on the set $G_R$ the process $B(Z^\varepsilon_\cdot) -B_R(Z^\varepsilon_\cdot)$ coincides with the process $\rho(\|Z^\varepsilon_\cdot\|_V)(B(Z^\varepsilon_\cdot) -B_R(Z^\varepsilon_\cdot))$, which has the advantage of being a progressively measurable process in $L^\infty(\Omega,L^\infty([0,T],H))$, so that \autoref{lem:trick} applies. Substituting into (\ref{eq:cap}), we obtain
 \begin{align*}
\Probe{E_{\varepsilon^{-2}\gamma}^c \cap G_R} 
&\leq 
\Probe{\left| \varepsilon^{-1}\int_0^T  \langle \rho(\|Z^\varepsilon_s\|_V)(B(Z^\varepsilon_s) -B_R(Z^\varepsilon_s)) ,dW_s \rangle \right| > \frac{\varepsilon^{-2}\gamma}{2} }
\\
&\leq 
2 \exp \left( -\frac{\varepsilon^{-2}\gamma^2 R^4}{8C^2T}\right).
\end{align*}

Putting all together we obtain uniformly in $\varepsilon, R$, with ${ 2 T C (a+\lambda_0^{-\alpha}bR+1)} < {\gamma R^2}$:
\begin{align*}
\Probe{X^\varepsilon \in U_\delta} &\geq \Probe{X^{\varepsilon,R} \in U_\delta} e^{-\varepsilon^{-2} \gamma} \\
&- C \exp\left(\frac{\varepsilon^{-2}T(a+\lambda_0^{-\alpha}bR+1)^2}{2} - \frac{c \varepsilon^{-2} (R-\|x_0\|_{V_0})^2}{2}  \right) \\
&-2 \exp \left( \frac{\varepsilon^{-2}T(a+\lambda_0^{-\alpha}bR+1)^2}{2} -\frac{\varepsilon^{-2}\gamma^2 R^4}{16C^2T}\right).
\end{align*}

Now we fix $R=R_{\gamma,\varphi}$ large enough such that the following inequalities hold
\begin{gather*}
{ 2 T C (a+\lambda_0^{-\alpha}bR+1)} < {\gamma R^2}, \quad S_{B_R}(\varphi) < S_{B}(\varphi) + \gamma, \\
\frac{T(a+\lambda_0^{-\alpha}bR+1)^2}{2} -\frac{c(R-\|x_0\|_{V_0})^2}{2}  < -S_{B}(\varphi) - 5\gamma,\\
\frac{T(a+\lambda_0^{-\alpha}bR+1)^2}{2} - \frac{\gamma^2 R^4}{16C^2T} < -S_{B}(\varphi) - 5\gamma.
\end{gather*}

Notice that here arises the condition on $T$, since the third inequality can be satisfied only for $T<c\lambda_0^\alpha/b^2$ (recall that the constant $c$ is non increasing in $T$, hence it is always possible to find such a $T$). However, in the case of $B\in C_b(H,H)$, this additional condition does not appear, since one can take $b=0$ and the third inequality can be always satisfied for $R$ large enough. With this choice of $R$ and using the fact that $\mathcal{L}( X^{\varepsilon,R})$ satisfies a LDP, we finally get for every $\varepsilon<\varepsilon_0$
\begin{align*}
\Probe{X^\varepsilon \in U_\delta} &\geq \exp \left( -\varepsilon^{-2} [S_{B}(\varphi) + 3\gamma]\right) 
- (2+C) \exp \left( -\varepsilon^{-2} [S_{B}(\varphi) + 5\gamma]\right) \\
&\geq \exp \left( -\varepsilon^{-2} [S_{B}(\varphi) + 4\gamma]\right),
\end{align*}
where $\varepsilon_0$ is chosen in such a way that both $\exp \left( -\varepsilon^{-2} 3\gamma\right) \geq 2 \exp \left( -\varepsilon^{-2} 4\gamma\right)$ and $(2+C) \exp \left( -\varepsilon^{-2} 5\gamma\right) \leq \exp \left( -\varepsilon^{-2} 4\gamma\right)$ for every $\varepsilon<\varepsilon_0$.
\end{proof}

\begin{prop} \label{prop:good}
Assume {(H)} and $B\in C(H,H)$ with at most linear growth. Then $S_B$ is an action functional on $C([0,T],V)$. 
\end{prop}

\begin{proof}
Closedness of the sublevels of $S_B$ is easy: indeed, take a sequence $\varphi^n \to \varphi^\infty$ in $C([0,T],V)$. Since any converging sequence is bounded, there exists $\bar{R}$ such that $\sup_{t \in [0,T]}\|\varphi^\infty\|_V \leq \bar{R}$ and $\sup_{t \in [0,T]}\|\varphi^n\|_V \leq \bar{R}$ for every $n \in \mathbb{N}$.
As a consequence, by \autoref{thm:approx} there exists some constant $C$  such that for every $R \geq \bar{R}$ and every $\varphi = \varphi^n$ or $\varphi = \varphi^\infty$:
\begin{align*}
S_B(\varphi) \leq S_{B_R}(\varphi) + C/R^2, \quad  S_{B_R}(\varphi) \leq S_{B}(\varphi) + C/R^2. 
\end{align*}
Since $S_{B_R}$ is an action functional it is lower semicontinuous, hence 
\begin{equation*}
S_B(\varphi^\infty) \leq  S_{B_R}(\varphi^\infty) + C/R^2
\leq \liminf_{n \to \infty} S_{B_R}(\varphi^n) + C/R^2 \leq \liminf_{n \to \infty} S_{B}(\varphi^n) + 2C/R^2, 
\end{equation*}
and closedness follows from the arbitrarity of $R$.

For the compactness we consider $K_M$ given by the definition of exponential tightness; if $\varphi \in K_M^c$, then by lower bound (\ref{eq:LowerBound'}) for any $\delta>0$ small enough:
\begin{align*}
-S_B(\varphi) - \gamma &\leq 
\liminf_{\varepsilon \to 0} \varepsilon^2 \log \Probe{ X^\varepsilon \in U_\delta} 
\\
&\leq \liminf_{\varepsilon \to 0} \varepsilon^2 \log \Probe{  X^\varepsilon \in K_M^c} < -M, 
\end{align*}
that implies $S_B(\varphi) \geq M$ by arbitrarity of $\gamma$. This means that $\Phi(s) \subseteq K_M$ for every $M>s$. Being a closed subset of a compact set, $\Phi(s) $ is compact as well.
\end{proof}

\begin{rmk}
Under the stronger assumptions $B \in C_b(H,H)$ and $x_0 \in D((-A^{1/2}))$, we have a simplified argument for proving that $S_B$ is an action functional on $C([0,T],V)$.
Indeed, by Simon compactness criterion \cite[Corollary 9]{Si86}, the embedding 
\begin{align*}
W^{1,2}([0,T],H) \cap L^2([0,T],D(A)) \subseteq C([0,T],V)  
\end{align*}
is compact, $V$ being equal to $D((-A)^{\alpha})$ with $\alpha \in (0,1/2)$; therefore, it is sufficient to prove that sublevels of $S_B$ are bounded subsets of $W^{1,2}([0,T],H) \cap L^2([0,T],D(A))$.
In order to see this, let us invoke again inequality (\ref{eq:inequality_rmk}) from \autoref{rmk:domain}:
\begin{align*}
\int_0^T &\left\|\dot{\varphi}_t\right\|_H^2 dt
+
\int_0^T \left\|A \varphi_t\right\|_H^2 dt
\\
&\leq \nonumber
\sup_{N \in \mathbb{N}}
\int_0^T \left\|\frac{d}{dt} (\pi_N \varphi_t) - A\pi_N \varphi_t\right\|_H^2 dt
+
2 \left\|x_0\right\|_{D((-A)^{1/2})}^2
\\
&\leq \nonumber
2\sup_{N \in \mathbb{N}}
\int_0^T \left\|\frac{d}{dt} (\pi_N \varphi_t) - A\pi_N\varphi_t - \pi_N B(\varphi_t)\right\|_H^2 dt
+
2 T\|B\|^2_{\infty}
+
2 \left\|x_0\right\|_{D((-A)^{1/2})}^2
\\
&\leq 4 S_B(\varphi)
+
2 T\|B\|^2_{\infty}
+
2 \left\|x_0\right\|_{D((-A)^{1/2})}^2,
\end{align*}
where the last line comes from the formula
\begin{align*}
S_B(\varphi) 
=
\sup_{N \in \mathbb{N}} S_N(\pi_N \psi)
=
\sup_{N \in \mathbb{N}} \frac12 \int_0^T \left\|\frac{d}{dt} (\pi_N \varphi_t) - A\pi_N \varphi_t  - \pi_N B(\varphi_t) \right\|_H^2 dt,
\end{align*}
proved in \autoref{lem:domain}.
Hence, assuming $\|B\|_{\infty}< \infty$ and $x_0 \in D((-A)^{1/2})$, we have proved that sublevels of $S_B$ are bounded subsets of $W^{1,2}([0,T],H) \cap L^2([0,T],D(A))$, hence compact in $C([0,T],V)$.

\end{rmk}

\begin{proof}[Proof of the upper bound] 
By \cite[Lemma 1.2.18]{DeZe09} it is sufficient to check (\ref{eq:UpperBound}) for any given compact $K\subseteq C([0,T],V)$. Arguing as in the proof of lower bound, one obtains
\begin{gather*}
\Probe{X^\varepsilon \in K} \leq \Probe{X^{\varepsilon,R} \in K} e^{\alpha} + \Probe{E_\alpha^c}^{1/2}\Ee{e^{2\xi_T^R}}^{1/2}.
\end{gather*}

Taking $\alpha = \varepsilon^{-2} \gamma$ and $R$ sufficiently large such that ${ 2 T C (a+\lambda_0^{-\alpha}bR+1)} < {\gamma R^2}$ we obtain uniformly in $\varepsilon$:
\begin{align*}
\Probe{X^\varepsilon \in K} &\leq \Probe{X^{\varepsilon,R} \in K} e^{\varepsilon^{-2} \gamma} \\
&+ C \exp\left(\frac{\varepsilon^{-2}T(a+\lambda_0^{-\alpha}bR+1)^2}{2} - \frac{c \varepsilon^{-2} (R-\|x_0\|_{V_0})^2}{2}  \right) \\
&+2 \exp \left( \frac{\varepsilon^{-2}T(a+\lambda_0^{-\alpha}bR+1)^2}{2} -\frac{\varepsilon^{-2}\gamma^2 R^4}{16C^2T}\right).
\end{align*}

Let $s_K= \inf_{\varphi \in K}S_B(\varphi)$. The case $s_K=\infty$ is the easy one, since one can make RHS in the equation above arbitrarily small using that $\mathcal{L}(X^{\varepsilon,R})$ satisfies a LDP, hence suppose $s_K<\infty$. If $s_K=0$ there is nothing to prove, otherwise take $R=R_{\gamma,K}$ sufficiently large such that the following inequalities hold
\begin{gather*}
{ 2 T C (a+\lambda_0^{-\alpha}bR+1)} < {\gamma R^2},  \\
\frac{T(a+\lambda_0^{-\alpha}bR+1)^2}{2} -\frac{c(R-\|x_0\|_{V_0})^2}{2}  < -2s_K,\\
\frac{T(a+\lambda_0^{-\alpha}bR+1)^2}{2} - \frac{\gamma^2 R^4}{16C^2T} < -2s_K.
\end{gather*}

With this choice of $R$ we get for every $\varepsilon<\varepsilon_0$
\begin{align*}
\Probe{X^\varepsilon \in K} &\leq \Probe{X^{\varepsilon,R} \in K} e^{\varepsilon^{-2} \gamma} + (2+C) \exp\left(-\varepsilon^{-2}2s_K \right) \\
&\leq \Probe{X^{\varepsilon,R} \in K} e^{\varepsilon^{-2} \gamma} +\exp\left(-\varepsilon^{-2}s_K \right),
\end{align*}
where $\varepsilon_0$ is chosen in such a way that $(2+C) \exp \left( -\varepsilon^{-2} 2 s_K \right) \leq \exp \left( -\varepsilon^{-2} s_K \right)$ for every $\varepsilon<\varepsilon_0$.

Now we use the fact that $\mathcal{L}(X^{\varepsilon,R})$ satisfies a LDP to estimate
\begin{align} \label{eq:lims}
\limsup_{\varepsilon\to 0} \varepsilon^2 \log \Probe{X^\varepsilon \in K} &\leq \max 
\left\{ \gamma + \limsup_{\varepsilon\to 0} \varepsilon^2 \log \Probe{X^{\varepsilon,R}\in K}, -s_K \right\} \\
&\leq \max \nonumber
\left\{ \gamma - \inf_{\varphi \in K}S_{B_R}(\varphi), -s_K \right\}.
\end{align}

By lower semicontinuity of $S_{B_R}$ and compactness of $K$, the infimum is attained at a certain $\varphi^R \in K$, namely $\inf_{\varphi \in K}S_{B_R}(\varphi) = S_{B_R}(\varphi^R)$. By compactness, there exists a subsequence (which we still denote by $\varphi^R$) such that $\varphi^R $ converges in $C([0,T],V)$ to a certain $\varphi^\infty \in K$. Moreover, the whole sequence $\varphi^R$ is uniformly bounded in $C([0,T],V)$, hence
\begin{align*}
\sup_{\substack{R\geq 1,\\ t \in [0,T]}} \|\varphi^R_t\|_V \leq \overline{R},
\end{align*}
with $\overline{R}$ sufficiently large. By \autoref{prop:good} $S_B$ is lower semicontinuous and therefore
\begin{align*}
S_B(\varphi^\infty) &\leq \liminf_{R \to \infty} S_B(\varphi^{R}) \\ 
&= \liminf_{R \to \infty} \frac{1}{2} \int_0^T
\left\| \dot{\varphi}^{R}_t - A \varphi^{R}_t - B_R(\varphi^{R}_t) + B_R(\varphi^{R}_t) -B(\varphi^{R}_t) \right\|_H^2 dt\\
&\leq \liminf_{R \to \infty} \left( (1+c^2) S_{B_{R}}(\varphi^{R}) + \frac{1 + c^{-2}}{2}\int_0^T \|B_R(\varphi^R_t) - B(\varphi^R_t)\|^2_H dt \right) \\
&\leq \liminf_{R \to \infty} \left( (1+c^2) S_{B_{R}}(\varphi^{R}) + \frac{(1 + c^{-2})TC^2}{2R^4} \right) = (1+c^2)\liminf_{R \to \infty} \left(  S_{B_{R}}(\varphi^{R}) \right).
\end{align*}

Plugging this inequality into (\ref{eq:lims}) and
taking $c \to 0$, $\gamma \to 0$ we finally obtain 
\begin{equation*}
\limsup_{\varepsilon\to 0} \varepsilon^2 \log \Probe{X^\varepsilon \in K} \leq -s_K = -\inf_{\varphi \in K}S_B(\varphi).
\end{equation*}
\end{proof}

\section{Applications} \label{sec:applications}

In this section we briefly discuss some particular equation to which our general result applies.

\subsection{Degenerate operator $A$}
As pointed out in \cite[Remark 2]{DPFlPrRo15}, \autoref{thm:LDP} still holds under the more general assumption that there exists $\omega \in \mathbb{R}$ such that the operator $A-\omega Id$ satisfies (H). Indeed one can rewrite (\ref{eq:SDE}) in the form
\begin{equation*}
dX^{\varepsilon}_t = \left( AX^{\varepsilon}_t - \omega X^{\varepsilon}_t \right)dt + \left( \omega X^{\varepsilon}_t + B(X^{\varepsilon}_t) \right) dt + \varepsilon dW_t, \quad t \in [0,T].
\end{equation*}

\subsection{Fractional diffusion equation}
Consider the \emph{fractional diffusion equation} on the one-dimensional torus $\mathbb{T}$
\begin{equation*}
\frac{\partial u}{\partial t}(t,x) + (-\Delta)^{\sigma} u(t,x) = 0, \quad (t,x) \in [0,T] \times \mathbb{T}, \quad \sigma > 0.
\end{equation*}

This PDE arises naturally when considering, for instance, the limiting behaviour of Boltzmann equation, see \cite{MMM11,JKO09}. Notice that in the sub-case $\sigma > 1/2$ the equation above satisfies the hypotheses of our abstract setting with $H = L^2(\mathbb{T})$, $A = -(-\Delta)^{\sigma}$ with domain $D(A)=H^{2\sigma}(\mathbb{T})$ and $B=0$.
Indeed, for every $\omega>0$ the operator $-(-\Delta)^{\sigma}-\omega$ is self-adjoint, strictly negative and $((-\Delta)^{\sigma}+\omega)^{-1+\delta}$ is trace-class for every $\delta< 1 - \frac{1}{2\sigma}$.
Therefore, we deduce the validity of a LDP on the space $C([0,T],H^{\alpha}(\mathbb{T}))$, $\alpha<\sigma - 1/2$, for the perturbed equation
\begin{equation*}
\frac{\partial u}{\partial t}(t,x) + (-\Delta)^{\sigma} u(t,x) = \varepsilon \frac{\partial \eta}{\partial t}(t,x), \quad (t,x) \in [0,T] \times \mathbb{T}, \quad \sigma > 1/2,
\end{equation*}
where $\frac{\partial \eta}{\partial t}$ is a space-time white noise and the initial condition is $u_0 \in H^{\sigma-1/2}(\mathbb{T})$, with rate $\varepsilon^2$ and action functional
\begin{equation*}
S(u) = \frac{1}{2} \int_0^T  \int_\mathbb{T} \left| \frac{\partial u}{\partial t}(t,x) + (-\Delta)^\sigma u(t,x) \right|^2 dxdt, 
\end{equation*}
if $u \in W^{1,2}([\tau,T],L^2(\mathbb{T})) \cap L^2([\tau,T],H^{2\sigma}(\mathbb{T}))$ for every $\tau>0$, $u(0,x)=u_0(x)$ and $S(u) = +\infty$ otherwise.

\subsection{Nonlinear fractional diffusion equation}
One natural extension of the fractional diffusion equation above is clearly its nonlinear counterpart:
\begin{equation*}
\frac{\partial u}{\partial t}(t,x) + (-\Delta)^{\sigma} u(t,x) = b(u(t,x)).
\end{equation*}

An interesting choice of $b$ is, for example, $b(u)=|u|^{\gamma}$, $\gamma \in (0,1)$, so that, as far as we know, the nonlinear equation above does not fall into the scope of any previous work concerning Large Deviations (cfr. \cite{CW99,CeRo04} for the locally-Lipschitz case). Quite remarkably, uniqueness of strong solutions to the unperturbed equation above does not hold in general, but also in this case a LDP holds on $C([0,T],H^{\alpha}(\mathbb{T}))$, $\alpha<\sigma - 1/2$, $\sigma > 1/2$, for the perturbed equation on $[0,T] \times \mathbb{T}$
\begin{equation*}
\frac{\partial u}{\partial t}(t,x) + (-\Delta)^{\sigma} u(t,x) = |u(t,x)|^{\gamma} + \varepsilon \frac{\partial \eta}{\partial t}(t,x),
\end{equation*}
with rate $\varepsilon^2$ and action functional
\begin{equation*}
S(u) = \frac{1}{2} \int_0^T  \int_\mathbb{T} \left| \frac{\partial u}{\partial t}(t,x) + (-\Delta)^\sigma u(t,x) - |u(t,x)|^{\gamma} \right|^2 dxdt, 
\end{equation*}
if $u \in W^{1,2}([\tau,T],L^2(\mathbb{T})) \cap L^2([\tau,T],H^{2\sigma}(\mathbb{T}))$ for every $\tau>0$, $u(0,x)=u_0(x)$ and $S(u) = +\infty$ otherwise. Notice that the action functional vanishes at $u$ if and only if $u$ is a strong solution to the unperturbed equation, a phenomenon similar to \cite{Herr01,GHR01} and \cite{Ma10}. 

\section*{Declaration of interests}
The author declares that he has no known competing financial interests or personal relationships that could have appeared to influence the work reported in this paper.

\section*{Acknowledgement}
The author is deeply grateful to Franco Flandoli for the useful discussions and for some of the ideas here exposed, and to the anonymous referees for their careful reading of the first version of this paper.


\bibliographystyle{plain}

\end{document}